\documentclass[11pt,a4paper]{article}

\usepackage[english]{babel}
\usepackage{times}
\usepackage{graphicx}
\usepackage{amscd}
\usepackage{amsmath}
\usepackage{amsfonts}
\usepackage{amssymb}
\usepackage{amsthm}
\usepackage{latexsym}

\newtheorem{theorem}{Theorem}

\newtheorem{lemma}[theorem]{Lemma}
\newtheorem{corollary}[theorem]{Corollary}

\newcommand{\R}{\mathbb{R}}

\renewcommand{\epsilon}{\varepsilon}
\newcommand{\eps}{\varepsilon}

\renewcommand{\le}{\leqslant}

\renewcommand{\ge}{\geqslant}

\begin{document}

\title{Rigidity of critical points \\ for a nonlocal Ohta-Kawasaki energy}

\author{Serena Dipierro\footnote{School of Mathematics and Statistics,
University of Melbourne, 813 Swanston St, Parkville VIC 3010, Australia,
and
Weierstra{\ss}-Institut f\"ur Angewandte
Analysis und Stochastik, Hausvogteiplatz 5/7, 10117 Berlin, Germany, serydipierro@yahoo.it}
\and
Matteo Novaga\footnote{Department of Mathematics, University of Pisa, novaga@dm.unipi.it}
\and
Enrico Valdinoci\footnote{School of Mathematics and Statistics,
University of Melbourne, 813 Swanston St, Parkville VIC 3010, Australia,
Weierstra{\ss}-Institut f\"ur Angewandte
Analysis und Stochastik, Hausvogteiplatz 5/7, 10117 Berlin, Germany, and
Dipartimento di Matematica, Universit\`a degli studi di Milano,
Via Saldini 50, 20133 Milan, Italy, enrico@math.utexas.edu} }

\date{} 

\maketitle

\begin{abstract}
We investigate the shape of critical points for a free energy consisting of a nonlocal perimeter plus a nonlocal repulsive term. In particular,  we prove that a volume-constrained critical point is necessarily a ball if its volume is sufficiently small with
respect to its isodiametric ratio, thus extending a result 
previously known only for global minimizers.

We also show that, at least in one-dimension, 
there exist critical points with arbitrarily small volume and large 
isodiametric ratio. This example shows that a constraint on the diameter is, in general, necessary to establish the radial symmetry of the critical points.
\end{abstract}

\tableofcontents

\section{Introduction}
In recent years the following variational problem 
has been widely studied:
\begin{equation}\label{eqcoul}
\min_{|E|=m} P(E)+
\iint_{E\times E} 
\frac{dx\,dy}{|x-y|},
\end{equation}
where $P(E)$ denotes the perimeter of a set $E\subset \R^3$.
Such problem first appeared in the liquid drop model of the atomic nuclei proposed by Gamow in 1928 
\cite{Gamow} and then developed by other researchers \cite{Bohr,CPS}, and it is
also relevant in some models of diblock copolymer melts \cite{CR,OK}.

In \cite{MK1,MK2} (see also \cite{CS}) the authors showed that global
minimizers of \eqref{eqcoul}
exist if the volume $m$ is small, and do not exist if the volume is large enough. 
They also showed that minimizers are necessarily balls if the volume is small enough.
This result has been later extended in \cite{Goldman} to connected critical points in two-dimensions.
However it is still not known if there exist non-spherical critical points of \eqref{eqcoul}
with arbitrarily small volume.

In this paper we consider the modified energy
\begin{equation}\label{eqF}
F(E):= P_s(E)+
\iint_{E\times E} 
\frac{dx\,dy}{|x-y|^\alpha}, 
\end{equation}
with $s\in (0,1)$ and $\alpha\in(0,n)$,
which is a nonlocal extension of the energy in \eqref{eqcoul}, which takes
into account long-range interactions in the attractive term
of surface tension type.

Here 
\[
P_s(E) :=\iint_{E\times E^c}\frac{dx\,dy}{|x-y|^{n+s}}
\]
denotes the so-called fractional perimeter of $E$, recently introduced by Caffarelli, Roquejoffre and Savin in \cite{CRS}, where they initiated the analysis of the corresponding Plateau Problem.

The functional \eqref{eqF} has been studied in \cite{I5}, where the authors showed that 
volume-constrained minimizers are balls, if the volume is small enough,
thus extending the results in \cite{MK1,MK2}.

We observe that, for any~$\lambda>0$, it holds
$$ F(\lambda E) = \lambda^{n-s}
\left( P_s(E) + \lambda^{n-\alpha+s}
\iint_{E\times E} 
\frac{dx\,dy}{|x-y|^\alpha}\right).$$
Therefore, minimizing~$F$ under the volume
constraint~$|E|=m$ is equivalent to
minimize the functional
$$ F_\eps(E) := P_s(E) +\eps
\iint_{E\times E}
\frac{dx\,dy}{|x-y|^\alpha},$$
under the constraint $|E|=1$,
with~$\eps:=m^{1-\frac \alpha n+\frac s n}$.

If~$E$ is a critical point of~$F_\eps$
with boundary of class~$C^2$,
then
\begin{equation}\label{EL} 
\kappa_E (x)+ c \eps V_E(x) = \lambda_\eps 
\qquad{\mbox{ for any }}x\in\partial E,\end{equation}
for a suitable~$\lambda_\eps\in\R$. Here above, $c>0$ is a normalizing constant,
\begin{equation}\label{VE} V_E (x):=\int_E
\frac{dy}{|x-y|^\alpha}
\end{equation}
is the potential, and~$\kappa_E$ is the fractional mean curvature of~$E$, i.e.,
\begin{equation}\label{K}
\kappa_E(x):=\int_{\R^n} \frac{\chi_{\R^n\setminus E}(y)-\chi_{E}(y)}{|x-y|^{n+s}}\,dy,
\end{equation}
where the integral is meant in the principal value sense,
see e.g.~\cite{CRS, ABA, I5} for further details on this notion of mean curvature. 

Notice that, if $E$ is a critical point of~$F_\eps$
with boundary of class~$C^2$, by elliptic regularity (see \cite{Barrios})
the boundary of $E$ is in fact of class $C^\infty$.

The main result of this paper is that {\em critical points of~$F$
are necessarily spherically symmetric, if their volume is small enough with respect to their 
isodiametric ratio}, which is defined as
\begin{equation}\label{iotas} 
I(E) := \frac{m^{\frac1n}}{{\rm diam}(E)}.
\end{equation}

To state this result precisely, we introduce a scale parameter
that links volume and diameter of a set, given by
\begin{equation}\label{betas} 
\beta=\beta(n,\alpha,s):=\frac{n+s-\alpha}{(2n+s+1)n}.\end{equation}

With this notation, we have:

\begin{theorem}\label{MAIN}
Let~$n\ge 2$, $s\in(0,1)$ and~$\alpha\in(0,n-1)$. 

There exists~$c_o=c_o(n,\alpha,s)>0$ such that, if
$E_\star$ is a bounded, volume-constrained, critical point of~$F$ with smooth boundary, such that 
%and with boundary of class~$C^{2,\gamma}$ for some $\gamma\in(s,1)$. 
\begin{equation}\label{HY}
|E_\star|^\beta \le c_o\; I(E_\star),
\end{equation}
then~$E_\star$ is a ball.

Equivalently, if
$E$ is a bounded, volume-constrained, critical point of~$F_\epsilon$ with smooth boundary, such that 
\begin{equation}\label{HYE}
|E|=1 \qquad and \qquad 
 {\rm diam}(E)\le  c_o\; \epsilon^{-\frac{1}{2n+s+1}},
\end{equation}
then~$E$ is a ball.
\end{theorem}

We observe that condition \eqref{HY}
does not violate the isodiametric inequality, since the right hand side is scale invariant, and so it is not void for small sets.
Moreover, we remark that
in problems like the ones considered in this paper, in which
an aggregating energy is competing with a repulsive one, universal bounds on
the diameter of the solutions do not hold in general. Indeed,
roughly speaking, the repulsive term of the energy functional may produce
critical points with disconnected components which lie ``far away''
one from the other
and which therefore violate uniform diameter bounds. To 
give a concrete example of this phenomenon we provide the following result:

\begin{theorem}\label{TH:EX}
Let~$n=1$, $s\in(0,1)$ and~$\alpha\in(0,1)$. 

There exists $\bar m=\bar m(\alpha,s)>0$ such that,
if~$m\in(0,\bar m)$,
there exists a bounded critical point $E_\star$ of $F$ with volume $m$,
made by two disconnected components, whose diameter satisfies
$$ {\rm diam}(E_\star) \ge C_o,$$
for some~$C_o>0$.

Equivalently, there exists~$\bar \epsilon=\bar \epsilon(\alpha,s)>0$ such that, if~$\epsilon\in(0,\bar \epsilon)$,
there exists a bounded set~$E\subset\R$
which is a solution of~\eqref{EL},
which is made by two disconnected components, whose diameter satisfies
$$ {\rm diam}(E) \ge C_o\,\epsilon^{-\frac{1}{1+s-\alpha}}.$$
In particular, the diameter of $E$ is not bounded uniformly in~$\varepsilon$.\end{theorem}

One of the main tools in the proof of Theorem~\ref{MAIN} is a nonlocal version of 
Alexandrov's Theorem, recently proved in \cite{Ciraolo2}, which states 
that regular sets with almost constant fractional mean curvature are necessarily balls.
This result can be applied to solutions to \eqref{EL}, after deriving suitable estimates on the potential $V_E$.

We point out that such quantitative Alexandrov's Theorem does not hold in the local setting,
where it is known that a set with almost constant mean curvature
is close to a family of tangent balls (see \cite{Ciraolo}), and this is the main obstruction 
for extending our result to the case $s=1$, that is, to the functional in \eqref{eqcoul}.

\medskip

The plan of the paper is the following: in Section~\ref{secest}
we provide some preliminary estimates on the potential $V_E$ and on its gradient,
which will be useful in the proof of Theorem~\ref{MAIN};
in Section~\ref{secth1} we prove  Theorem~\ref{MAIN};
finally in Section~\ref{secth2} we prove  Theorem~\ref{TH:EX},
giving an explicit example of a one-dimensional set, composed by two disjoint segments,
which is a critical point of \eqref{eqF} and has  arbitrarily small volume.

\section{Potential estimates}\label{secest}

In this section we provide some bounds on the potential~$V$ and on its
derivatives that will be used in the proof of the main results.

\begin{lemma}\label{Lal}
Let~$B$ be a ball centered at the origin with~$|B|=|E|$. Then, for any~$x\in \R^n$,
$$ V_E(x)\le V_B(0) \le C,$$
for some~$C>0$.
\end{lemma}

\begin{proof} Let~$\rho>0$ be the radius of~$B$. In this way, we have that~$x+B=B_\rho(x)$.
Also, if~$y\in B_\rho(x)\setminus E$, then~$|x-y|\le\rho$ and so
$$ \int_{B_\rho(x)\setminus E} \frac{dy}{|x-y|^{\alpha}}\ge
\int_{B_\rho(x)\setminus E} \frac{dy}{\rho^{\alpha}} = 
\frac{ |B_\rho(x)\setminus E|}{\rho^{\alpha}}.$$
Similarly, if~$y\in E\setminus B_\rho(x)$, then~$|x-y|\ge\rho$ and so
$$ \int_{E\setminus B_\rho(x)} \frac{dy}{|x-y|^{\alpha}}\le
\frac{ |E\setminus B_\rho(x)|}{\rho^{\alpha}}.$$
Moreover
$$ |E\setminus B_\rho(x)| =
|E|-|E\cap B_\rho(x)| = |B_\rho(x)|-|B_\rho(x)\cap E|=
|B_\rho(x)\setminus E|.$$
Consequently,
$$ \int_{E\setminus B_\rho(x)} \frac{dy}{|x-y|^{\alpha}}\le
\int_{B_\rho(x)\setminus E} \frac{dy}{|x-y|^{\alpha}}.$$
Therefore, summing the contributions in~$E\cap B_\rho(x)$ to both sides
of this inequality, we obtain that
$$ \int_{E} \frac{dy}{|x-y|^{\alpha}}\le
\int_{B_\rho(x)} \frac{dy}{|x-y|^{\alpha}}.$$
This and the integrability of the kernel imply the desired result.
\end{proof}

Following are additional potential estimates:

\begin{lemma}\label{LP:P}
We have that
\begin{eqnarray}\label{LAK:Au1}
&& \int_E \nabla V_E(x)\cdot x\,dx
= -\frac{\alpha}{2}\int_E V_E(x)\,dx
\\ {\mbox{and }}&&\label{LAK:Au2}
\int_{\partial E} V_E(x)\,x\cdot\nu(x)\,d{\mathcal{H}}^{n-1}(x)
=\left( n-\frac\alpha2\right)
\,\int_E  V_E(x)\,dx
.\end{eqnarray}
Moreover,
\begin{equation}\label{123}
|\nabla V_E(x)|
\le C,
\end{equation}
for any~$x\in\R^n$, for a suitable~$C>0$.\end{lemma}

\begin{proof} We observe that
\begin{equation}\label{AJ:101} \nabla V_E(x)=
-\alpha\int_E \frac{(x-y)\,dy}{|x-y|^{\alpha+2}},\end{equation}
so we can write
$$ \beta:=
\int_E \nabla V_E(x)\cdot x\,dx
= -\alpha \iint_{E\times E} 
\frac{(x-y)\cdot x}{|x-y|^{\alpha+2}}\,dx\,dy.$$
Since the role played by the variable in the latter
integral is symmetric, we can also write
$$ \beta=
-\alpha \iint_{E\times E} 
\frac{(y-x)\cdot y}{|x-y|^{\alpha+2}}\,dx\,dy.$$
As a consequence
\begin{eqnarray*}
&& 2\beta=-\alpha \iint_{E\times E} 
\frac{(x-y)\cdot x}{|x-y|^{\alpha+2}}\,dx\,dy
-\alpha \iint_{E\times E} 
\frac{(y-x)\cdot y}{|x-y|^{\alpha+2}}\,dx\,dy\\ &&\qquad\qquad=
-\alpha \iint_{E\times E} 
\frac{(x-y)\cdot (x-y)}{|x-y|^{\alpha+2}}\,dx\,dy
=-\alpha \iint_{E\times E} 
\frac{dx\,dy}{|x-y|^{\alpha}}.\end{eqnarray*}
This establishes~\eqref{LAK:Au1}.

Now we use the Divergence Theorem and~\eqref{LAK:Au1} to see that
\begin{eqnarray*}
&& \int_{\partial E} V_E(x)\,x\cdot\nu(x)\,d{\mathcal{H}}^{n-1}(x)
= \int_E {\rm div}\,\big( V_E(x)\,x\big)\,dx
\\ &&\qquad\qquad=
n \int_E  V_E(x)\,dx+
\int_E \nabla V_E(x)\cdot x\,dx
\\ &&\qquad\qquad=\left( n-\frac\alpha2\right)
\,\int_E  V_E(x)\,dx,\end{eqnarray*}
and so we have proved~\eqref{LAK:Au2}.

Finally, by~\eqref{AJ:101},
$$ |\nabla V_E(x)|\le 
\alpha\int_E \frac{dy}{|x-y|^{\alpha+1}},$$
hence~\eqref{123} follows as in Lemma~\ref{Lal} (replacing~$\alpha$ there
with~$\alpha+1\in(0,n)$).
\end{proof}

With this, we can show that smooth critical points
possess a uniform bound on~$\lambda_\eps$
in the Euler-Lagrange equation in~\eqref{EL}:

\begin{corollary}
There exist~$C_1$, $C_2>0$ such that
$$ \big|\lambda_\eps - C_1\,P_s(E)\big|\le C_2\eps.$$
\end{corollary}

\begin{proof}
Given a vector field~$ X \in C^\infty_0(\R^n)$,
we denote, for any small~$t>0$,
by~$\Phi^t$ the flow
induced by~$ X$, as defined by the Cauchy problem
$$ \begin{cases}
\partial_t \Phi^t(x)= X(\Phi^t(x)),\\
\Phi^0(x)=x.
\end{cases}$$
Then (up to normalizing constants that we neglect,
see e.g. Formula~(6.12) in~\cite{I5}), we have that
\begin{equation} \label{JA:coam}
\frac{\partial}{\partial t} P_s(\Phi^t(E))\Big|_{t=0}
=\int_{\partial E} \kappa_E(x)\,X(x)\cdot\nu(x)\,d{\mathcal{H}}^{n-1}(x),\end{equation}
being~$\nu$ the external normal of~$E$.

Now, by taking~$X(x):=x$ in 
the vicinity of~$E$, 
we have that
$$\Phi^t(x)=x+tX(x)+O(t^2)=(1+t)x+O(t^2)$$
and so~$|\Phi^t(E)|=(1+t+O(t^2))\,|E|$.
Therefore, by scaling, we have that
$$P_s(\Phi^t(E))=(1+t+O(t^2))^{n-s}\,P_s(E)$$
and thus
$$ \frac{\partial}{\partial t} P_s(\Phi^t(E))\Big|_{t=0}=
(n-s)\,P_s(E).$$
Comparing this and~\eqref{JA:coam}, we obtain
\begin{equation}\label{8ihS} 
P_s(E)= \int_{\partial E} \kappa_E(x)\,x\cdot\nu(x)\,d{\mathcal{H}}^{n-1}(x),\end{equation}
up to multiplicative normalization constants that we neglect.

Also, by the Divergence Theorem,
$$ \int_{\partial E} x\cdot\nu(x)\,d{\mathcal{H}}^{n-1}(x)
= \int_E {\rm div}\,x\,dx= n\,|E|.$$
Hence, by~\eqref{EL}, \eqref{8ihS} and Lemma~\ref{LP:P},
up to multiplicative constants,
\begin{equation}\begin{split}\label{uyhjHA}
\lambda_\eps \,|E|
\;&=\;\int_{\partial E} \lambda_\eps x\cdot\nu(x)\,d{\mathcal{H}}^{n-1}(x)
\\&=\;\int_{\partial E} \kappa_E(x)\,x\cdot\nu(x)\,d{\mathcal{H}}^{n-1}(x)
+\eps\int_{\partial E} V_E(x)\,x\cdot\nu(x)\,d{\mathcal{H}}^{n-1}(x)\\
&=\; P_s(E)+\eps\int_E V_E(x)\,dx.
\end{split}\end{equation}
Also, by Lemma~\ref{Lal},
$$ \int_E V_E(x)\,dx \le C\,|E|.$$
{F}rom this, \eqref{uyhjHA} and the volume constraint~$|E|=1$,
the desired result plainly follows.
\end{proof}

Now we control the tangential gradient of~$V_E$:

\begin{lemma}\label{HMA}
Let~$x\in\partial E$ and~$\tau_{E}(x)$ be a tangential vector to~$\partial E$ at~$x$.
Assume that~$E=T(B)$,
for a suitable ball~$B$, with~$|B|=|E|$ and let~$T$ be a diffeomorphism such that
$$ \| T-{\rm Id}\|_{C^1(B)}\le \mu,$$
for some~$\mu\ge0$.
Then, there exists~$\mu_0>0$ such that for any~$\mu\in[0,\mu_0]$
\begin{equation}\label{787y} \big|\nabla V_E(x)\cdot\tau_E(x)\big|
\le C\mu,\end{equation}
for some~$C>0$.
\end{lemma}

\begin{proof} We define~$\bar x:=T^{-1}(x)$
and
$$ \tau_B(\bar x):=DT^{-1}(x)\,\tau_E(x)=
\big(DT(\bar x)\big)^{-1}\,\tau_E(x)
.$$ Notice
that~$\tau_B(\bar x)$ is a tangent vector to~$\partial B$ at~$\bar x\in\partial B$.
Also, by rotational symmetry, the function~$V_B$ is constant along~$\partial B$,
and therefore
\begin{equation}\label{oiKALL:1} \nabla V_B(\bar x)\cdot \tau_B(\bar x)=0.\end{equation}
Moreover, by~\eqref{AJ:101} and using the substitution~$\bar y:=T^{-1}(y)$,
\begin{equation}\label{oiKALL:2}\begin{split}
\nabla V_E(x)\cdot \tau_E(x) \,&=
-\alpha\int_E \frac{(x-y)\cdot \tau_E(x)}{|x-y|^{\alpha+2}} \,dy
\\ &=
-\alpha\int_B \frac{(T(\bar x)-T(\bar y))\cdot \big( DT(\bar x)
\tau_B(\bar x)\big)}{|T(\bar x)-T(\bar y)|^{\alpha+2}} 
\,|\det DT(\bar y)|
\,d\bar y.\end{split}
\end{equation}
Now, by~\eqref{787y},
\begin{equation}\label{89aA}\begin{split}
& \Big| DT(\bar x)\tau_B(\bar x) - \tau_B(\bar x) \Big|\le C\mu,\\
& \Big| \det DT(\bar y) -1\Big|\le C\mu,\\
{\mbox{and }}\quad&
\Big| T(\bar x)-T(\bar y) -(\bar x-\bar y) \Big|
\le
\int_0^1 \big|DT(t\bar y+(1-t)\bar x)\big|\,dt\,|x-y|
\le C\mu\,|\bar x-\bar y|
\end{split}\end{equation}
for some~$C>0$, and so
$$ 
(1-C\mu)\,|\bar x-\bar y|\le \big| T(\bar x)-T(\bar y) \big|
\le (1+C\mu)\,|\bar x-\bar y|,$$
which in turn implies that
\begin{eqnarray*}
&& \left| \frac{1}{|T(\bar x)-T(\bar y)|^{\alpha+2}} -
\frac{1}{|\bar x-\bar y|^{\alpha+2}} \right|
=\frac{1}{|\bar x-\bar y|^{\alpha+2}} \,
\left| 1 -
\frac{|\bar x-\bar y|^{\alpha+2}}{{|T(\bar x)-T(\bar y)|^{\alpha+2}}} \right|\le
\frac{C\mu}{|\bar x-\bar y|^{\alpha+2}}
,\end{eqnarray*}
up to renaming~$C>0$.
Using this and~\eqref{89aA} inside~\eqref{oiKALL:2}, we obtain that
\begin{eqnarray*}
&& \frac{1}{\alpha}\,\Big| \nabla V_E(x)\cdot \tau_E(x) - \nabla V_B(\bar x)\cdot \tau_B(\bar x)\Big|\\
&=&\left|
\int_B \frac{(T(\bar x)-T(\bar y))\cdot \big( DT(\bar x)
\tau_B(\bar x)\big)}{|T(\bar x)-T(\bar y)|^{\alpha+2}} 
\,|\det DT(\bar y)|
\,d\bar y
-\int_B \frac{(\bar x-\bar y)\cdot
\tau_B(\bar x)}{|\bar x-\bar y|^{\alpha+2}} 
\,d\bar y \right| \\
&\le& 
\left|
\int_B \frac{(T(\bar x)-T(\bar y))\cdot \big( DT(\bar x)
\tau_B(\bar x)\big)}{|T(\bar x)-T(\bar y)|^{\alpha+2}} 
\,|\det DT(\bar y)|
\,d\bar y
\right.\\ &&\quad \left. 
- \int_B \frac{(\bar x-\bar y)\cdot \big( DT(\bar x)
\tau_B(\bar x)\big)}{|T(\bar x)-T(\bar y)|^{\alpha+2}} 
\,|\det DT(\bar y)|
\,d\bar y\right|\\ &&\quad
+\left| \int_B \frac{(\bar x-\bar y)\cdot \big( DT(\bar x)
\tau_B(\bar x)\big)}{|T(\bar x)-T(\bar y)|^{\alpha+2}} 
\,|\det DT(\bar y)|
\,d\bar y
\right.\\ &&\quad \left. 
-
\int_B \frac{(\bar x-\bar y)\cdot 
\tau_B(\bar x)}{|T(\bar x)-T(\bar y)|^{\alpha+2}} 
\,|\det DT(\bar y)|
\,d\bar y\right|
\\ &&\quad+
\left| 
\int_B \frac{(\bar x-\bar y)\cdot 
\tau_B(\bar x)}{|T(\bar x)-T(\bar y)|^{\alpha+2}} 
\,|\det DT(\bar y)|
\,d\bar y-
\int_B \frac{(\bar x-\bar y)\cdot 
\tau_B(\bar x)}{|\bar x-\bar y|^{\alpha+2}} 
\,|\det DT(\bar y)|
\,d\bar y\right|\\ &&\quad+
\left|
\int_B \frac{(\bar x-\bar y)\cdot 
\tau_B(\bar x)}{|\bar x-\bar y|^{\alpha+2}} 
\,|\det DT(\bar y)|
\,d\bar y
-\int_B \frac{(\bar x-\bar y)\cdot
\tau_B(\bar x)}{|\bar x-\bar y|^{\alpha+2}} 
\,d\bar y \right|
\\ &\le& C\mu
\int_B \frac{d\bar y}{|\bar x-\bar y|^{\alpha+1}} 
\\ &\le&C\mu,
\end{eqnarray*}
possibly changing~$C>0$ from one line to another.
{F}rom this estimate and~\eqref{oiKALL:1},
the desired result plainly follows.
\end{proof}

\section{Proof of Theorem~\ref{MAIN}}\label{secth1}

By a spatial dilation, we replace $E_\star$ with a critical point $E$ of the functional~$F_\eps$,
under the volume constraint~$|E|=1$. To this aim, we define
$$ E:= m^{-\frac1n}\,E_\star \quad{\mbox{ and }}\quad \eps:=m^{1-\frac \alpha n+\frac s n}.$$
In particular ${\rm diam}(E)=m^{-\frac1n}\,{\rm diam}(E_\star)$ and thus,
by \eqref{iotas} and \eqref{betas}, condition \eqref{HY} becomes
\begin{equation}\label{epsis}
\begin{split}&
{\eps}\,{\big( {\rm diam}(E)\big)^{2n+s+1}}=\left( {\eps^{\frac{1}{2n+s+1}}\cdot m^{-\frac1n}}{
\,{\rm diam}(E_\star)}\right)^{2n+s+1} \\ 
&\qquad\quad= \left( { m^{\frac{n-\alpha+s}{(2n+s+1)n}}\cdot m^{-\frac1n} }{
\,{\rm diam}(E_\star)}\right)^{2n+s+1} = \left( \frac{m^\beta}{ I(E_\star)}\right)^{2n+s+1} \le c_o^{2n+s+1},\end{split}\end{equation}
which is supposed to be a small quantity.

Also, the Euler-Lagrange equation in~\eqref{EL}
holds true. 
Therefore we have that
\begin{equation}\label{deltas}
\begin{split}
\delta_s(E)\,&:=\sup_{{x,y\in\partial E}\atop{x\ne y}}
\frac{|\kappa_E(x)-\kappa_E(y)|}{|x-y|} \\
&=\sup_{{x,y\in\partial E}\atop{x\ne y}}
\frac{\Big| \big(c\,\eps\,V_E(x)-\lambda_\eps\big)-
\big(c\,\eps\,V_E(y)-\lambda_\eps\big)
\Big|}{|x-y|} \\
&=c\eps\,\sup_{{x,y\in\partial E}\atop{x\ne y}}
\frac{|V_E(x)-V_E(y)|}{|x-y|} .
\end{split}
\end{equation}
Thus, by \eqref{123},
\begin{equation}\label{deltas2}
\delta_s(E)\le c\eps\,\sup_{{x,y\in\R^n}\atop{x\ne y}}
\frac{|V_E(x)-V_E(y)|}{|x-y|} \le C\eps,
\end{equation}
for some $C>0$.
Now, letting 
\begin{equation}\label{tfygHAJ}
\eta_s(E) := \big({\rm diam}(E)\big)^{2n+s+1}\,\delta_s(E),
\end{equation}
by Formula~(1.4) in~\cite{Ciraolo2} we have that
\begin{equation}\label{321}
\rho(E):=
\inf\left\{ \frac{R-r}{{\rm diam}(E)}, {\mbox{ with }}p\in E,\; B_r(p)\subseteq E\subseteq B_R(p)\right\}
\le C\,\eta_s(E)
,\end{equation}
up to renaming~$C>0$. Also, by \eqref{deltas2}
and \eqref{tfygHAJ} we get
\begin{equation}\label{etastim}
\eta_s(E) \le C\,\big({\rm diam}(E)\big)^{2n+s+1}\,\eps.
\end{equation}
This and \eqref{epsis} give that $\eta_s(E)$ is small if so is $c_o$, hence we are in the position to apply Theorem 1.5
in \cite{Ciraolo2}. In particular, we obtain that $\partial E=F(\partial B)$, where $B$ is a ball with volume $m$, with
$$ \| F-{\rm Id}\|_{C^2(\partial B)}\le C\,\eta_s(E),$$
up to renaming $C>0$. That is, if we set $T(x):=|x|\,F\left(\frac{x}{|x|}\right)$, we have that $E=T(B)$
and
$$ \| T-{\rm Id}\|_{C^1(B)}\le C\,\eta_s(E).$$
This says that we are in the setting of Lemma~\ref{HMA}, with $\mu:= C\,\eta_s(E)$.
As a consequence, we obtain that
\begin{equation*}
\big|\nabla_\tau V_E(x)\big|
\le C\eta_s(E),\end{equation*}
for some~$C>0$. Hence, we insert this into \eqref{deltas} and we find
$$ \delta_s(E)\le C\,\eps\,\eta_s(E),$$
up to renaming $C>0$. Accordingly, in view of \eqref{tfygHAJ} and \eqref{epsis},
$$ \delta_s(E)\le C\,\eps\,\big(  {\rm diam}(E)\big)^{2n+s+1}\delta_s(E)\le C\,c_o^{2n+s+1}\,\delta_s(E)\le \frac{
\delta_s(E)}{2},$$
if $c_o$ is small enough. This implies that $\delta_s(E)=0$.

So, by \eqref{tfygHAJ}, we have that $\eta_s(E)=0$ and then, by \eqref{321}, also $\rho(E)=0$, which says that $E$ is a ball, as desired.

\section{Proof of Theorem~\ref{TH:EX}}\label{secth2}

We let
$$ E:=\left( 0,\,\frac12\right) \cup \left( d,\,d+\frac12\right),$$
for some~$d>1/2$, that we choose appropriately in order to
fulfill~\eqref{EL}.

By~\eqref{K} and~\eqref{VE}, for any~$x\in \{0,\;1/2,\; d,\; d+1/2\}$,
\begin{eqnarray*}
\zeta_E(x)&:=& \kappa_E(x)+c\varepsilon \,V_E(x)\\
&=& \int_{-\infty}^0 
\frac{dy}{|x-y|^{1+s}}
- \int_{0}^{1/2} \frac{dy}{|x-y|^{1+s}}
+ \int_{1/2}^d\frac{dy}{|x-y|^{1+s}}
\\&&\qquad- \int_{d}^{d+1/2} 
\frac{dy}{|x-y|^{1+s}}+\int_{d+1/2}^{+\infty}
\frac{dy}{|x-y|^{1+s}} \\
&&\qquad + c\varepsilon \,\int_{0}^{1/2} \frac{dy}{|x-y|^{\alpha}}
+ c\varepsilon \,\int_{d}^{d+1/2} \frac{dy}{|x-y|^{\alpha}}.\end{eqnarray*}
By symmetry (i.e. using the substitution~$\tilde y:= (d+1/2)-y$),
we see that
\begin{equation}\label{ZETA.1}
\zeta_E(0)=\zeta_E\left( d+\frac12\right)\quad{\mbox{ and }}\quad
\zeta_E\left(\frac12\right)=\zeta_E(d).
\end{equation}
Now we prove that we can choose~$d>1/2$ in such a way that
\begin{equation}\label{ZETA.2}
\zeta_E(0)=\zeta_E\left(\frac12\right).
\end{equation}
To this aim, we compute
\begin{eqnarray*}
f(d)&:=& \zeta_E\left(\frac12\right)-\zeta_E(0)\\
&=& \int_{-\infty}^0
\frac{dy}{|y-1/2|^{1+s}}
- \int_{0}^{1/2} \frac{dy}{|y-1/2|^{1+s}}
+ \int_{1/2}^d\frac{dy}{|y-1/2|^{1+s}}
\\&&\qquad- \int_{d}^{d+1/2}
\frac{dy}{|y-1/2|^{1+s}} 
+\int_{d+1/2}^{+\infty}
\frac{dy}{|y-1/2|^{1+s}}
\\
&&\qquad -\int_{-\infty}^0
\frac{dy}{|y|^{1+s}}
+\int_{0}^{1/2} \frac{dy}{|y|^{1+s}}
-\int_{1/2}^d\frac{dy}{|y|^{1+s}}\\&&\qquad+
\int_{d}^{d+1/2}\frac{dy}{|y|^{1+s}} -
\int_{d+1/2}^{+\infty}
\frac{dy}{|y|^{1+s}}
\\
&&\qquad + c\varepsilon \,\int_{0}^{1/2} \frac{dy}{|y-1/2|^{\alpha}}
+ c\varepsilon \,\int_{d}^{d+1/2} \frac{dy}{|y-1/2|^{\alpha}}
\\&&\qquad -c\varepsilon \,\int_{0}^{1/2} \frac{dy}{|y|^{\alpha}}
-c\varepsilon \,\int_{d}^{d+1/2} \frac{dy}{|y|^{\alpha}}\\
&=& \int_{-\infty}^{-1/2}
\frac{dy}{|y|^{1+s}}
- \int_{-1/2}^{0} \frac{dy}{|y|^{1+s}}
+ \int_{0}^{d-1/2}\frac{dy}{|y|^{1+s}}
- \int_{d-1/2}^{d}
\frac{dy}{|y|^{1+s}} 
+\int_{d}^{+\infty}
\frac{dy}{|y|^{1+s}}\\
&&\qquad -\int_{-\infty}^0
\frac{dy}{|y|^{1+s}}
+\int_{0}^{1/2} \frac{dy}{|y|^{1+s}}
-\int_{1/2}^d\frac{dy}{|y|^{1+s}}+
\int_{d}^{d+1/2}\frac{dy}{|y|^{1+s}} 
-\int_{d+1/2}^{+\infty}
\frac{dy}{|y|^{1+s}}
\\
&&\qquad + c\varepsilon \,\int_{-1/2}^{0} \frac{dy}{|y|^{\alpha}}
+ c\varepsilon \,\int_{d-1/2}^{d} \frac{dy}{|y|^{\alpha}}
-c\varepsilon \,\int_{0}^{1/2} \frac{dy}{|y|^{\alpha}}
-c\varepsilon \,\int_{d}^{d+1/2} \frac{dy}{|y|^{\alpha}}\\
&=& -2\int_{d-1/2}^d\frac{dy}{|y|^{1+s}}
+2\int_{d}^{d+1/2}\frac{dy}{|y|^{1+s}} 
+ c\varepsilon \,\int_{d-1/2}^{d} \frac{dy}{|y|^{\alpha}}
-c\varepsilon \,\int_{d}^{d+1/2} \frac{dy}{|y|^{\alpha}}
\\ &=&
\frac{2}{s}\left[ 2d^{-s}-(d-1/2)^{-s}-(d+1/2)^{-s}\right]
\\ &&\qquad
+ \frac{c\varepsilon}{1-\alpha}
\left[ 2d^{1-\alpha}-(d-1/2)^{1-\alpha}- (d+1/2)^{1-\alpha}\right]
.\end{eqnarray*}
Now we use that, for any~$\beta\in\R$ and~$\delta>0$ small,
$$ 2-(1-\delta)^\beta-(1+\delta)^\beta =
-\beta\,(\beta-1)\,\delta^2 +O(\delta^3),$$
therefore, for large~$d$,
\begin{equation}\label{FG}\begin{split}
f(d)\;&=\,
\frac{2d^{-s}}{s}\left[ 2-\left(1-\frac{1}{2d}\right)^{-s}
-\left(1+\frac{1}{2d}\right)^{-s}\right]
\\ &\qquad
+ \frac{c\varepsilon\,d^{1-\alpha}}{1-\alpha}
\left[ 2- \left(1-\frac{1}{2d}\right)^{1-\alpha}-
\left(1+\frac{1}{2d}\right)^{1-\alpha}\right]
\\ &=\,
\frac{2d^{-s}}{s}\cdot\left(\frac{-s(1+s)}{4d^2}+O(d^{-3})\right)
+\frac{c\varepsilon\,d^{1-\alpha}}{1-\alpha}\cdot
\left(\frac{\alpha(1-\alpha)}{4d^2}+O(d^{-3})\right)
\\ &=\, d^{-1-\alpha}
\left( 
\frac{c\alpha\varepsilon}{4}
-\frac{(1+s)d^{-1-s+\alpha}}{2}
+O(d^{-2-s+\alpha})
+O(\varepsilon d^{-3})
\right)
\\ &=\, d^{-1-\alpha}\,g(d),
\end{split}\end{equation}
where
$$ g(d):= \frac{c\alpha\varepsilon}{4}
-\frac{(1+s)d^{-1-s+\alpha}}{2}
+O(d^{-2-s+\alpha})
+O(\varepsilon d^{-3}).$$
Notice that
\begin{equation}\label{G.1} 
g(+\infty)=\frac{c\alpha\varepsilon}{4} >0.\end{equation}
Also, if
$$ d_\varepsilon:= \left(\frac{1+s}{c\alpha\varepsilon}\right)^{\frac{1}{1+s-\alpha}},$$
we have that
$$ g(d_\varepsilon)= -\frac{c\alpha\varepsilon}{4}
+O(\varepsilon^{1+\frac{1}{1+s-\alpha}})+O(\varepsilon^{1+\frac{3}{1+s-\alpha}})
\le -\frac{c\alpha\varepsilon}{8} <0,$$
if~$\varepsilon$ is sufficiently small.

{F}rom this and~\eqref{G.1} we obtain that there exists
an appropriate~$d>d_\varepsilon$ for which~$g(d)=0$,
and thus, recalling~\eqref{FG},
we obtain that~$f(d)=0$, which proves~\eqref{ZETA.2}.

Now, from~\eqref{ZETA.1} and~\eqref{ZETA.2}, we deduce that
$$ \zeta_E(0)=\zeta_E\left( d+\frac12\right)=
\zeta_E\left(\frac12\right)=\zeta_E(d)=:\lambda,$$
for some~$\lambda\in\R$, that is
$$ \kappa_E(x)+c\varepsilon \,V_E(x)=\lambda$$
for any~$x\in\partial E$, as desired.

\section*{Acknowledgements}
Part of this work was carried out while Serena Dipierro
and Enrico Valdinoci were visiting the University of Pisa, which
they wish to thank for the hospitality.

This work has been supported by the Alexander von Humboldt Foundation, the
ERC grant 277749 {\it E.P.S.I.L.O.N.} ``Elliptic
Pde's and Symmetry of Interfaces and Layers for Odd Nonlinearities'',
the PRIN grant 201274FYK7
``Aspetti variazionali e
perturbativi nei problemi differenziali nonlineari''
and the University of Pisa grant PRA- 2015-0017.

Finally, the authors wish to thank Giulio Ciraolo  and Cyrill Muratov for 
useful conversations on the subject of this paper.

\end{document}